\documentclass[letterpaper, 10 pt, conference]{ieeeconf}  

\IEEEoverridecommandlockouts                              

\usepackage{cite}
\usepackage{amsmath,amssymb,amsfonts}
\usepackage{graphicx} 
\usepackage{amsmath}
\usepackage{amssymb}  
\usepackage{mathrsfs}
\usepackage{multicol}
\usepackage{xcolor}
\usepackage{float}
\usepackage{subcaption}
\usepackage{mathtools}
\usepackage{soul}
\usepackage{bm}
\usepackage[T1]{fontenc}
\newtheorem{rem}{Remark}
\newtheorem{Assumption}{Assumption}

\newtheorem{prop}{Proposition}

\newtheorem{theorem}{Theorem}

\usepackage{tikz}
\usepackage{xcolor}
\usepackage{dsfont}
\usepackage{algorithm}
\usepackage{algpseudocode}
\usepackage{hyperref} 

\let\labelindent\relax
\usepackage[shortlabels]{enumitem}

\DeclareMathAlphabet\mathbfcal{OMS}{cmsy}{b}{n}

\definecolor{darkspringgreen}{rgb}{0.09, 0.45, 0.27}

\def\BibTeX{{\rm B\kern-.05em{\sc i\kern-.025em b}\kern-.08em
    T\kern-.1667em\lower.7ex\hbox{E}\kern-.125emX}}

\newcommand{\cG}{\mathcal{G}}
\newcommand{\cE}{\mathcal{E}}
\newcommand{\cV}{\mathcal{V}}
\newcommand{\cN}{\mathcal{N}}

\newcommand{\Erre}[0]{\mathbfcal{R}}
\newcommand{\hatbf}[1]{\widehat{\mathbf{#1}}}
\def\1{\mathds{1}}

\title{\LARGE \bf
Optimal distributed control with stability guarantees \\ by training a network of neural closed-loop maps
}

\author{Danilo Saccani, Leonardo Massai, Luca Furieri and Giancarlo Ferrari-Trecate
\thanks{This research has been supported by the Swiss National Science Foundation under the NCCR Automation (grant agreement 51NF40\_180545).}
\thanks{The authors are with the Institute of Mechanical Engineering, Ecole Polytechnique Fédérale de Lausanne (EPFL), CH-1015 Lausanne, Switzerland. (email: \tt\small {\{l.massai,danilo.saccani,luca.furieri, giancarlo.ferraritrecate\}@epfl.ch) } }%
}

\begin{document}

\maketitle
\thispagestyle{empty}
\pagestyle{empty}

\begin{abstract}

This paper proposes a novel approach to improve the performance of distributed nonlinear control systems while preserving stability by leveraging Deep Neural Networks (DNNs). We build upon the Neural System Level Synthesis (Neur-SLS) framework and introduce a method to parameterize stabilizing control policies that are distributed across a network topology. A distinctive feature is that we iteratively minimize an arbitrary control cost function through an unconstrained optimization algorithm, all while preserving the stability of the overall network architecture by design. This is achieved through two key steps. First, we establish a method to parameterize interconnected Recurrent Equilibrium Networks (RENs) that guarantees a bounded $\mathcal{L}_2$ gain at the network level. This ensures stability. Second, we demonstrate how the information flow within the network is preserved, enabling a fully distributed implementation where each subsystem only communicates with its neighbors. To showcase the effectiveness of our approach, we present a simulation of a distributed formation control problem for a fleet of vehicles. The simulation demonstrates how the proposed neural controller enables the vehicles to maintain a desired formation while navigating obstacles and avoiding collisions, all while guaranteeing network stability.

\end{abstract}


\section{Introduction}
Controlling interconnected systems with non-linear dynamics is a crucial task across various fields ranging from robotics, and power grids, to social sciences~\cite{liu2011controllability,antonelli2013interconnected,barrat2008dynamical,sinopoli2003distributed}. However, achieving networked control becomes significantly more challenging when the system dynamics are non-linear and a convex function does not accurately capture the desired performance metric.

Indeed, several Optimal Distributed Control (ODC) approaches exist for the case of linear systems with convex performance metrics \cite{rotkowitz2005characterization,lessard2011quadratic, wang2019system, furieri2019separable,furieri2020sparsity}; however, for non-convex problems, solutions are scarce. Dissipativity-based approaches, like passivity control \cite{van2000l2,arcak2016networks} and control by interconnection \cite{ortega2008control}, offer ways to design stabilizing distributed control policies without minimizing a specific cost. However, these methods become prohibitive when it comes to minimizing a user-defined cost function. Nonlinear Distributed Model Predictive Control (DMPC) \cite{scattolini2009architectures,liu2009distributed} gives an alternative, but computing crucial terminal ingredients in a distributed manner remains unsolved for non-linear systems. Additionally, suitable distributed optimization algorithms for online implementation are limited, often relying on the Alternating Direction Method of Multipliers (ADMM) \cite{nishihara2015general}, whose potential lack of convergence poses limitations in real-time applications.

Recently, there has been a surge of interest in using Deep Neural Networks (DNNs) to parameterize highly non-linear distributed controllers. Approaches like Graph Neural Networks (GNNs) have shown promise in designing such policies \cite{tolstaya2020learning,gama2021graph} in a scalable way. Nevertheless, stability guarantees are often limited by restrictive assumptions. For instance, some methods require a system to be initially linear and open-loop stable, with the trained network maintaining small Lipschitz constants across layers \cite{gama2021graph}. These conditions might not hold during exploration, potentially leading to system failures before a safe policy is learned \cite{cheng2019end}.

An alternative approach using deep Hamiltonian control policies offers built-in closed-loop stability guarantees for specific systems \cite{furieri2022distributed}. However, these policies only apply to a narrow class of systems with a particular skew-symmetric interconnection structure.


\subsubsection*{Contributions}
With this work, we provide two main contributions. First, we propose a novel distributed state-feedback control scheme leveraging the Neural System Level Synthesis (Neur-SLS) framework and dissipativity theory. The scheme is composed of a parametrized family of nonlinear sub-operators, interconnected according to a prescribed network structure that mimics the sparsity pattern of the distributed system itself. This means that communication between sub-operators only occurs between interconnected agents in the physical system.  We prove that closed-loop stability is guaranteed as long as the parameters satisfy a Linear Matrix Inequality (LMI).
Secondly, we present a \textit{free} parametrization of this networked ensemble of nonlinear sub-operators, with each sub-operator modeled as a Recurrent Equilibrium Network (REN). The parameterization is \textit{free} in the sense that the resulting distributed controller guarantees the closed-loop stability for any value of the parameters. This eliminates the need to enforce LMIs, thereby enabling the training of a distributed controller using unconstrained learning algorithms akin to neural networks. This sparse neural controller structure facilitates distributed implementation across individual subsystems while preserving communication topology among the subsystems.



\section*{Notation}
Throughout the paper, we denote with $\mathbb{N}$ the set of non-negative integers.
$I_{a}$ ($0_{a}$) is the identity (zero) square matrix with dimensions $a\times a$, $0_{a\times b}$ is the zero matrix with dimension $a\times b$ and $\1$ is the vector of all ones, where the dimension is clear from the context. We denote with diag$(v)$ a square diagonal matrix with the elements of vector $v$ on the main diagonal, while blkdiag$(A,B)$ denotes a block diagonal matrix created by aligning the matrices $A$ and $B$ along the diagonal.
Positive semidefinite matrices $A$ are denoted as $A \succeq 0$. The set of all sequences $\mathbf{v}=(v_0,v_1,v_2,\dots)$, where $v_t\in\mathbb{R}^n$ for all $t\in \mathbb{N}$, is denoted as $l^n$. Moreover, $\mathbf{v}$ belongs to $l^n_2\subset l^n$ if $\| \mathbf{v} \|_2=\left( \sum_{t=0}^{\infty} |v_t|^2 \right)^{\frac{1}{2}}<\infty$. We refer to $ \mathbf{v_T}$ to denote a truncation of $\mathbf{v}$ with $t$ ranging from $0$ to $T$. An operator $\mathbf{A}:l^n \rightarrow l^m$ is said to be \textit{causal} if $\mathbf{A}(\mathbf{x})=(A_0(\mathbf{x_0}), A_1(\mathbf{x_1}),\dots,A_t(\mathbf{x_t}),\dots)$. If $A_t(\mathbf{x_t})=A_t(0,\mathbf{x_{t-1}}),$ then $\mathbf{A}$ is said to be \textit{strictly} causal. 
An operator $\mathbf{A}:l^n\rightarrow l^m$ is said to be $\mathcal{L}_2$ stable if it is causal and $\mathbf{A}(\mathbf{a})\in l^m_2$ for all $\mathbf{a}\in l^n_2$. Equivalently, we write $\mathbf{A}\in\mathcal{L}_2$.  Given an undirected graph~$\cG=(\cV, \cE)$ described by the set of nodes $\cV$ and the set of edges $\cE\subset \cV \times \cV$, we denote set of neighbors of node~$i$, including~$i$ itself by $\mathcal{N}_i = \{i\} \cup \{j\ |\ \{i,j\}\in\cE\} \subseteq \cV$. We denote with col$_{j\in\mathcal{V}}(v^{[j]})$ a vector which consists of the stacked
subvectors $v^{[j]}$ from $j=1$ to $j=|\cV|$ and with $v^{[\mathcal{N}_i]}$ a vector composed by the stacked subvectors $v^{[i]}$, i.e. $v^{[\mathcal{N}_i]}=col_{i\in\mathcal{N}_i}(v^{[i]})$.
\section{Preliminaries}
\subsection{Distributed nonlinear systems}
We consider a network of $N$ interconnected nonlinear subsystems. The coupling network among the subsystems is defined as an undirected communication graph~$\cG=(\cV, \cE)$, with $\cV=\{1, \dots, N\}$ representing the subsystems in the network, and the set of edges~$\cE$ containing the pairs of subsystems~$\{i,j\}$, which can communicate with each other. Each subsystem is of the form
\begin{equation}
    x_t^{[i]}=f^{[i]}(x^{[\cN_i]}_{t-1},u^{[i]}_{t-1})+w^{[i]}_t,  \ \ \   \forall t\geq1, \label{eq:nonlinsyslocal}
\end{equation}
where state and input of each subsystem $i\in\mathcal{V}$ are denoted by $x^{[i]}\in\mathbb{R}^{n_i}$ and $u^{[i]} \in \mathbb{R}^{m_i}$ respectively. Moreover, $w^{[i]} \in \mathbb{R}^{n_i}$ is an unknown process noise with $w^{[i]}_0=x^{[i]}_0$.\\
In operator form, we can express the subsystem~\eqref{eq:nonlinsyslocal} as:
\begin{equation} \label{eq:sub-operator}
    \mathbf{x}^{[i]} =\mathbf{F}^{[i]}(\mathbf{x}^{[\mathcal{N}_i]},\mathbf{u}^{[i]})+\mathbf{w}^{[i]},
\end{equation}
where $\mathbf{F}^{[i]}:l^{n_{\mathcal{N}_i}}\times l^{m_i} \rightarrow l^{n_i}$ is a strictly causal operator defining the subsystem dynamics such that $\mathbf{F^{[i]}}(\mathbf{x^{[\mathcal{N}_i]},u^{[i]}}) = \left(0,f^{[i]}\left(x^{[\mathcal{N}_i]}_0,u_0^{[i]}\right),\dots,f^{[i]}\left(x^{[\mathcal{N}_i]}_t,u^{[i]}_t\right),\dots\right)$.\\
By combining the local system dynamics in~\eqref{eq:nonlinsyslocal}, the dynamics of the global system result in
\begin{equation} 
    x_t=f(x_{t-1},u_{t-1})+w_t,  \ \ \   \forall t\geq0, \label{eq:nonlinsys}
\end{equation}
where $x=col_{i\in\mathcal{V}}(x^{[i]})\in\mathbb{R}^{n}$, $u=col_{i\in\mathcal{V}}(u^{[i]})\in\mathbb{R}^{m}$,$w=col_{i\in\mathcal{V}}(w^{[i]})\in\mathbb{R}^{n}$.
Similarly to the subsystems, we can rewrite system~\eqref{eq:nonlinsys} in operator form as:
\begin{equation} \label{eq:operator}
    \mathbf{x} =\mathbf{F}(\mathbf{x,u})+\mathbf{w},
\end{equation}
We make the following Assumption on the system to be controlled:
\begin{Assumption}
\label{ass1}
    We assume that the causal operator $\mathbf{F}$ is such that the map $(\mathbf{w},\mathbf{u})\mapsto \mathbf{x}$ lies in $\mathcal{L}_2$, and that the process noise $w(t)\sim \mathcal{D}$ is distributed according to an unknown distribution $\mathcal{D}$, and that $\mathbf{w}$ belongs to $l^n_2$.
\end{Assumption}

The Assumption above means that the interconnection is stable or locally controlled to achieve $\mathcal{L}_2$ stability. 
This is typically true in practical applications where systems are either pre-stabilized or can be stabilized using simple controllers. While these controllers ensure stability, they often result in suboptimal closed-loop performance. Our goal is to enhance performance while maintaining stability.

To control the system, we consider nonlinear state-feedback causal policies parametrized as follows
\begin{equation} \label{eq:policy}
\mathbf{u}=\mathbf{K(x,\mathbf{u})}=(K_0(x_0),\dots,K_t(\mathbf{x}_{\mathbf{t}},\mathbf{u}_{\mathbf{t-1}}),\dots).
\end{equation}
Thanks to causality, each sequence of disturbances, denoted as $\mathbf{w}\in l^n_2$, results in unique trajectories for the closed-loop system~\eqref{eq:operator}-~\eqref{eq:policy}. This implies that the mappings from disturbances to state and control, denoted as $\mathbf{\Phi^x}[\mathbf{F,K}]$ and $\mathbf{\Phi^u}[\mathbf{F,K}]$ respectively, are well-defined for any given system $\mathbf{F}$ and controller $\mathbf{K}$. Thus, we express $\mathbf{x=\Phi^x}[\mathbf{F,K}]\mathbf{(w)}$ and $\mathbf{u=\Phi^u}[\mathbf{F,K}]\mathbf{ (w)}$ for all $\mathbf{w}\in l_2^n$.
\subsection{Problem formulation}
Our goal is to address the following challenges in designing a control policy $\mathbf{K}(\mathbf{x},\mathbf{u})$:
\begin{enumerate}[{\textbf{C}}-1]
    \item Ensuring the stability of the closed-loop mappings $\mathbf{\Phi^x}[\mathbf{F,K}]$ and $\mathbf{\Phi^u}[\mathbf{F,K}]$ in the $\mathcal{L}_2$ sense.
    \item Enhancing the performance of the closed-loop system by minimizing a loss function defined as:
    \begin{equation} \label{eq:loss}
    J= \mathbb{E}_{\mathbf{w} \sim \mathcal{D}} \left[ \mathbf{\ell}(\mathbf{x_T,u_T}) \right] \:,
    \end{equation}
    where $\mathbf{\ell}$ is differentiable and $\mathbf{\ell}(\mathbf{x_T,u_T})\geq 0$ for all $T\in\mathbb{N}$ and all $(x,u)\in\mathbb{R}^{n+m}$.
    \item Ensuring that the resulting policy $\mathbf{K}(\mathbf{x},\mathbf{u})$ can be implemented in a distributed manner exploiting a neighbor-to-neighbor communication. Specifically, we require that 
    \begin{equation}
   \mathbf{u}^{[i]} = \mathbf{K}^{[i]}(\mathbf{x}^{[\mathcal{N}_i]},\mathbf{u}^{[\mathcal{N}_i]}),\quad \forall i\in \mathcal{V}\label{eq:distributed_constraint}\,.
    \end{equation}
\end{enumerate}
\begin{rem}
Unlike most optimal control and dynamic programming problems that rely on a sum of stage costs over time, our approach is more flexible. We only require the loss function to be differentiable in its arguments. This allows us to incorporate complex objectives typical of many robotics and reinforcement learning tasks, such as logical specifications formulated using Signal Temporal Logic (STL)~\cite{leung2023backpropagation}. 
\end{rem}

We can now define the  finite horizon ODC problem we aim to solve:
\begin{subequations}\label{eq:NLPproblem}
\begin{alignat}{3}
    &\min_{\mathbf{K(\cdot)}} \quad && \mathbb{E}_{\mathbf{w} \sim \mathcal{D}} \left[ \mathbf{\ell}(\mathbf{x_T,u_T}) \right] \label{seq:costODC} \\
    &\operatorname{subject~to} \quad && w_0 = x_0, \quad \eqref{eq:nonlinsyslocal},~\eqref{eq:policy},~\eqref{eq:distributed_constraint},~\forall i \in \mathcal{V}, \label{seq:distrConstr}\\
    &&& (\mathbf{\Phi^x[F,K],\Phi^u[F,K]}) \in \mathcal{L}_2 \label{seq:globalStab}\:,
\end{alignat}
\end{subequations}
where~\eqref{eq:policy},~\eqref{eq:distributed_constraint} impose the causality and sparsity structure of the obtained policy while ensuring the satisfaction of~\eqref{seq:globalStab} guarantees the closed-loop stability of the distributed system at a global network level.

\begin{rem}
    Under Assumption \ref{ass1}, the problem \eqref{eq:NLPproblem} is always feasible (consider for example $\mathbf{K=0})$. 
\end{rem}

\begin{rem}
    Critically, \eqref{seq:globalStab} enforces strict closed-loop stability at the network level for all optimized distributed policies. However, in this paper we consider a stronger requirement since we are interested in iteratively solving the optimization problem~\eqref{eq:NLPproblem} while ensuring \textit{fail-safe learning}. This means remaining within the feasible region defined by~\eqref{seq:distrConstr} and~\eqref{seq:globalStab} at every iteration, while improving the obtained policy by minimizing the loss function~\eqref{seq:costODC}.
\end{rem}

\section{Main result: Free parametrization of distributed stable neural closed-loop maps}

\begin{figure}
    \centering
\tikzset{every picture/.style={line width=0.75pt}} 
\begin{tikzpicture}[x=0.75pt,y=0.75pt,yscale=-1,xscale=0.8]
\draw   (241,41) -- (349.75,41) -- (349.75,120) -- (241,120) -- cycle ;
\draw   (280.57,130.43) -- (320.57,130.43) -- (320.57,169.43) -- (280.57,169.43) -- cycle ;
\draw    (349.4,80) -- (386.75,80) ;
\draw [shift={(389.75,80)}, rotate = 180] [fill={rgb, 255:red, 0; green, 0; blue, 0 }  ][line width=0.08]  [draw opacity=0] (10.72,-5.15) -- (0,0) -- (10.72,5.15) -- (7.12,0) -- cycle    ;
\draw    (368.71,149.86) -- (368.31,80.4) ;
\draw    (368.71,149.86) -- (324.75,149.99) ;
\draw [shift={(321.75,150)}, rotate = 359.83] [fill={rgb, 255:red, 0; green, 0; blue, 0 }  ][line width=0.08]  [draw opacity=0] (10.72,-5.15) -- (0,0) -- (10.72,5.15) -- (7.12,0) -- cycle    ;
\draw    (205.24,80.08) -- (236,80.01) ;
\draw [shift={(239,80)}, rotate = 179.87] [fill={rgb, 255:red, 0; green, 0; blue, 0 }  ][line width=0.08]  [draw opacity=0] (10.72,-5.15) -- (0,0) -- (10.72,5.15) -- (7.12,0) -- cycle    ;
\draw   (194.51,80.08) .. controls (194.51,77.12) and (196.92,74.71) .. (199.88,74.71) .. controls (202.84,74.71) and (205.24,77.12) .. (205.24,80.08) .. controls (205.24,83.04) and (202.84,85.44) .. (199.88,85.44) .. controls (196.92,85.44) and (194.51,83.04) .. (194.51,80.08) -- cycle ;
\draw    (200.43,150.43) -- (200.58,89.87) ;
\draw [shift={(200.59,86.87)}, rotate = 90.14] [fill={rgb, 255:red, 0; green, 0; blue, 0 }  ][line width=0.08]  [draw opacity=0] (10.72,-5.15) -- (0,0) -- (10.72,5.15) -- (7.12,0) -- cycle    ;
\draw    (280.75,150.5) -- (200.43,150.43) ;
\draw    (90,80) -- (135.51,80.07) ;
\draw [shift={(138.51,80.08)}, rotate = 180.09] [fill={rgb, 255:red, 0; green, 0; blue, 0 }  ][line width=0.08]  [draw opacity=0] (10.72,-5.15) -- (0,0) -- (10.72,5.15) -- (7.12,0) -- cycle    ;
\draw  [dash pattern={on 4.5pt off 4.5pt}] (117.75,29.5) -- (440.25,29.5) -- (440.25,179.5) -- (117.75,179.5) -- cycle ;
\draw   (389.07,60.43) -- (429.07,60.43) -- (429.07,99.43) -- (389.07,99.43) -- cycle ;
\draw    (429.4,80) -- (459.25,80) ;
\draw [shift={(462.25,80)}, rotate = 180] [fill={rgb, 255:red, 0; green, 0; blue, 0 }  ][line width=0.08]  [draw opacity=0] (10.72,-5.15) -- (0,0) -- (10.72,5.15) -- (7.12,0) -- cycle    ;
\draw   (139.57,60.43) -- (179.57,60.43) -- (179.57,99.43) -- (139.57,99.43) -- cycle ;
\draw    (180.74,80.08) -- (191.25,80.02) ;
\draw [shift={(194.25,80)}, rotate = 179.67] [fill={rgb, 255:red, 0; green, 0; blue, 0 }  ][line width=0.08]  [draw opacity=0] (10.72,-5.15) -- (0,0) -- (10.72,5.15) -- (7.12,0) -- cycle    ;
\draw (242,46.6) node [anchor=north west][inner sep=0.75pt]    {${\textstyle \Erre^{[1]}}(v^{[1]})$};
\draw (268,93.9) node [anchor=north west][inner sep=0.75pt]    {$\Erre^{[N]}(v^{[N]})$};
\draw (279.2,64.4) node [anchor=north west][inner sep=0.75pt]    {$\ddots $};
\draw (284.27,142.83) node [anchor=north west][inner sep=0.75pt]    {$M_{vz}$};
\draw (374,123.9) node [anchor=north west][inner sep=0.75pt]    {$\mathbf{z}$};
\draw (213.29,60.17) node [anchor=north west][inner sep=0.75pt]    {$\mathbf{v}$};
\draw (449.9,52) node [anchor=north west][inner sep=0.75pt]    {$\boldsymbol{u}$};
\draw (87.3,58.57) node [anchor=north west][inner sep=0.75pt]    {$\hatbf{w}$};
\draw (392.77,72.83) node [anchor=north west][inner sep=0.75pt]    {$M_{uz}$};
\draw (90,154.9) node [anchor=north west][inner sep=0.75pt]    {$\Erre$};
\draw (142.27,72.83) node [anchor=north west][inner sep=0.75pt]    {$M_{vw}$};
\end{tikzpicture}
    \caption{Interconnection of $N$ operators $\mathbf{\Erre}^{[i]}$.}
    \label{fig:interconnectionREN}
\end{figure}
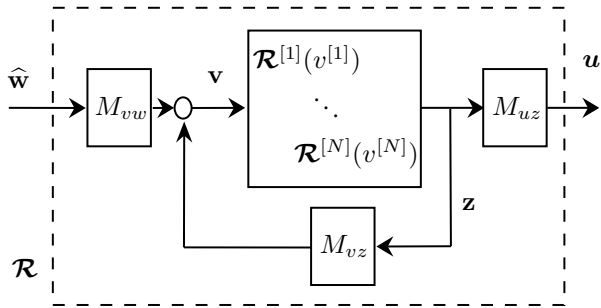
In this section, we will address the three main objectives we aim to achieve. Firstly, we will begin by recalling the process of parametrizing stabilizing control policies for the closed-loop system to establish a framework capable of satisfying \textbf{C}-1 with a distributed structure \textbf{C}-3. Then, we will demonstrate how a proper parametrization of the controller enables us to fulfill \textbf{C}-2 in an unconstrained fashion.

\subsection{Networked controller} \label{sec:InterconOper}
To achieve a distributed controller ensuring closed-loop stability, we first recall the findings introduced in \cite{furieri2022neural} for centralized control. 
Using insights from internal model control~\cite{economou1986internal}, the main result of \cite{furieri2022neural} is that under Assumption~\ref{ass1}, \textit{all and only} stabilizing control policies for the system $\mathbf{F}$ can be parameterized in terms of one stable operator $\Erre\in\mathcal{L}_2$ by implementing the control input $\mathbf{u}=\Erre(\hatbf{w})$
where $\hatbf{w} = \mathbf{x}-\mathbf{F(x,u)}$ reconstructs the true disturbances $\mathbf{w}$. The resulting closed loop system is given by
\begin{align} 
&\mathbf{x}=\mathbf{F(x,u)}+\mathbf{w}, \nonumber\\ 
&\mathbf{u}=\Erre(\mathbf{x}-\mathbf{F(x,u)})\,, \quad \Erre \in \mathcal{L}_2\,,\label{eq:closedloopglobal}
\end{align}
where $\Erre$ defines an overall control policy $\mathbf{K}$ such that $\mathbf{u}=\mathbf{K}(\mathbf{x},\mathbf{u}) = \Erre(\hatbf{w})$. 

Beyond achieving closed-loop stability, the information constraints \eqref{eq:distributed_constraint} of the ODC problem \eqref{eq:NLPproblem} bring in additional challenges. Indeed, implementing a controller $\mathbf{u} = \Erre(\hatbf{w})$ requires, at each node $i \in \mathcal{V}$, information about other $\hatbf{w}^{[i]}$'s, which in turn are reconstructed based on the neighboring states $x^{[\mathcal{N}_i]}$'s and $u^{[i]}$'s, that is, $\mathbf{w}^{[i]}=\mathbf{x}^{[i]}-\mathbf{F}^{[i]}(\mathbf{x}^{[\mathcal{N}_i]},\mathbf{u}^{[i]})$. This raises the question of how the constraints in~\eqref{eq:distributed_constraint} impact the design of $\Erre$.

Our solution is based on designing $N$ networked $\mathcal{L}_2$ sub-operators $\Erre^{[i]}: l^{q_i} \rightarrow l^{r_i}$ whose interconnection structure is shown in Fig.~\ref{fig:interconnectionREN} and resulting in the overall local controller shown in Fig. \ref{fig:localContr}.
Specifically, we aim to design sub-operators  with input $v^{[i]}\in\mathbb{R}^{q_i}$ and output $z^{[i]}\in\mathbb{R}^{r_i}$ with $ q_i \geq n_i$ and $r_i \geq m_i$, $\forall i \in \mathcal{V}$ whose state-space definition is as follows:
\begin{align}
    \xi^{[i]}_{t+1}= &\rho^{[i]}\left(\xi^{[i]}_{t},v^{[i]}_t \right) \, \notag \\
    z^{[i]}_t = & \chi^{[i]} \left(\xi^{[i]}_t,v^{[i]}_t\right) \: , \quad \forall i \in \mathcal{V}\,   &\forall t\geq0,\label{eq:nonlinsuboperlocal} 
\end{align}
where $\xi^{[i]}_t\in\mathbb{R}^{c}$ is the internal state.
We assume that each sub-operator has a finite $\mathcal{L}_2$-gain not grater than $\gamma^{[i]}$, meaning that there exists a non-negative local storage function $V^{[i]}$ such that for all $t\in\mathbb{N}$
\begin{equation}
\label{eq:diss}
    V^{[i]}(\xi^{[i]}_{t+1})-V^{[i]}(\xi^{[i]}_t)\leq s^{[i]}(v^{[i]}_t,z^{[i]}_t)\,,
\end{equation}
where $s^{[i]}(v^{[i]},z^{[i]})$ is the quadratic supply rate
\begin{equation}\label{eq:supplyREN}
    s^{[i]}(v^{[i]},z^{[i]})= \left[\begin{array}{l}
        v^{[i]}\\
         z^{[i]} 
    \end{array}\right]^{\top}
\underbrace{\left[\begin{array}{ll}
\gamma^{[i]^2} I_{q_i} & 0_{q_i\times r_i}\\
0_{r_i\times q_i} & -I_{r_i}
\end{array}\right]  }_{X^{[i]}}  
        \left[\begin{array}{l}
        v^{[i]}\\
         z^{[i]} 
    \end{array}\right].
\end{equation}
It is worth mentioning that~\eqref{eq:diss} is the definition of dissipativity for the system~\eqref{eq:nonlinsuboperlocal} with respect to the generic supply rate $s^{[i]}(v^{[i]},z^{[i]})$. For the specific supply given in~\eqref{eq:supplyREN}, we retrieve the well-known concept of finite $\mathcal{L}_2$-gain property, which is indeed a particular case of dissipativity (see~\cite{koelewijn2021incremental}).

The interconnection structure of sub-operators shown in Fig.~\ref{fig:interconnectionREN} is formalized as follows:
\begin{equation} \label{eq:interconnUE}
    \left[\begin{array}{l}
v \\
u
\end{array}\right]=\bar{M}\left[\begin{array}{l}
z \\
\widehat{w}
\end{array}\right]=\left[\begin{array}{ll}
M_{vz} & M_{vw} \\
M_{uz} & 0_{m\times n}
\end{array}\right]\left[\begin{array}{l}
z \\
\widehat{w}
\end{array}\right],
\end{equation}
where $v\in\mathbb{R}^{q}$ are the inputs to local controllers and $z\in\mathbb{R}^{r}$ the outputs.
The matrix $\bar{M}$ describes how each sub-operator input $v^{[i]}$ is coupled with each sub-operator output $z^{[i]}$ as well as with the exogenous input $w^{[i]}$ and output $u^{[i]}$.
Specifically, the matrix $M_{vz}$ describes how the input of each sub-operator is linked to the outputs of the others. In practice, $M_{vz}$ is defined coherently with the network structure defined by the graph $\mathcal{G}$ so that $v_i$ depends on $z_{\mathcal{N}_i}$.  

To obtain an overall operator $\Erre$ with a finite $\mathcal{L}_2$ gain $\gamma_R>0$ between $\widehat{w}$ and $u$, we need to guarantee the dissipativity of the interconnected operator $\Erre$, with respect to the quadratic supply rate:
\begin{equation} \label{eq:performanceSupply}
    \left[\begin{array}{l}
\widehat{w} \\
u
\end{array}\right]^\top 
\underbrace{\left[\begin{array}{cc}
\gamma_{R}^2 I_n & 0_{n\times m } \\
0_{m\times n} & -I_m
\end{array}\right]}_S
\left[\begin{array}{l}
\widehat{w} \\
u
\end{array}\right].
\end{equation}
We are now in the position to state the following proposition guaranteeing the $\mathcal{L}_2 $ stability of the closed-loop \eqref{eq:closedloopglobal}.
\begin{prop}
\label{propstab}
    Consider the closed-loop system given by~\eqref{eq:closedloopglobal} where the controller operator $\Erre(\widehat{\mathbf{w}})$ is composed of sub-operators interconnected according to~\eqref{eq:interconnUE}. Let Assumption \ref{ass1} hold, then the system \eqref{eq:closedloopglobal}  is $\mathcal{L}_2$ stable, i.e., $(\mathbf{\Phi^x[F,K],\Phi^u[F,K]})\in \mathcal{L}_2$ if
    there exist scalars $\alpha^{[i]}\geq0, i \in \cV$ such that the following inequality  holds:
\begin{equation} \label{eq:conditionL22}
\resizebox{.9\hsize}{!}{$
    \left[\begin{array}{ll}
        M_{vz} & M_{vw}\\
         I & 0 \\
         0&I \\
         M_{uz} & 0 
    \end{array}\right]^\top
    \left[\begin{array}{cc}
    \mathbf{X}(\alpha^{[i]}X^{[i]}) & 0 \\
    0 & -S
    \end{array}\right]
    \left[\begin{array}{ll}
        M_{vz} & M_{vw}\\
         I & 0 \\
         0&I \\
         M_{uz} & 0 
    \end{array}\right] \preceq 0 \:,$}
\end{equation}
where $\mathbf{X}(\alpha^{[i]}X^{[i]})$ is a matrix composed of the weighted matrices that define the local supply functions $X^{[i]}$ (see Appendix~\ref{appendix}), and the matrices $M_{vz}$, $M_{vw}$, $M_{uz}$ and $S$ are defined in~\eqref{eq:interconnUE} and~\eqref{eq:conditionL22}.
\end{prop}

\begin{proof}
    Thanks to Assumption \ref{ass1}, from~\cite[Thm.~2]{furieri2022neural} we have that  $(\mathbf{\Phi^x[F,K],\Phi^u[F,K]})\in \mathcal{L}_2$ for any $\Erre \in \mathcal{L}_2$.
    Now we have to prove that the operator $\Erre$ composed by interconnecting $\mathcal{L}_2$ stable operators $\Erre^{[i]}$ as in~\eqref{eq:interconnUE} has a finite $\mathcal{L}_2$ gain if the inequality~\eqref{eq:conditionL22} is satisfied. Following \cite{arcak2016networks}, we consider as candidate storage function the weighted sum of the local storage functions for the whole interconnection $
    V(\xi)=\alpha^{[1]} V^{[1]}\left(\xi^{[1]}\right)+\cdots+\alpha^{[N]} V^{[N]}\left(\xi^{[N]}\right)$, with $\alpha^{[i]} \geq 0$ and $\xi =col_{i\in\mathcal{V}}(\xi^{[i]})$. 
Thus, to obtain a finite $\mathcal{L}_2$ gain between $\widehat{w}$ and $u$ we want that the interconnection is dissipative with respect to the weighted sum of the supply rates~\eqref{eq:supplyREN} and that this term is dominated by the performance supply rate~\eqref{eq:performanceSupply}, i.e.,
\begin{multline}\label{eq:dissipativityInterconnection}
    \sum_{i=1}^N \alpha^{[i]} \left[ V^{[i]}(\xi^{[i]}_{t+1})- V^{[i]}(\xi^{[i]}_t) \right] \leq \\ \sum_{i=1}^N \alpha^{[i]} \left[\begin{array}{l}
        v^{[i]}\\
         z^{[i]}
    \end{array}\right]^{\top}
X^{[i]}
        \left[\begin{array}{l}
       v^{[i]}\\
         z^{[i]}
    \end{array}\right]\leq \\ \left[\begin{array}{l}
\widehat{w} \\
u
\end{array}\right]^\top
\underbrace{\left[\begin{array}{cc}
\gamma_{R}^2 I_n & 0_{n\times m } \\
0_{m\times n} & -I_m
\end{array}\right]}_S
\left[\begin{array}{l}
\widehat{w} \\
u
\end{array}\right] \leq 0.
\end{multline}
By rearranging the right-hand side of~\eqref{eq:dissipativityInterconnection} we can rewrite it as:
\begin{equation}
    \left[\begin{array}{l}
        v \\
        z\\
        \widehat{w} \\
        u
    \end{array}\right]^\top
    \left[\begin{array}{cc}
    \mathbf{X}(\alpha^{[i]}X^{[i]}) & 0 \\
    0 & -S
    \end{array}\right]
    \left[\begin{array}{l}
        v \\
        z\\
        \widehat{w} \\
        u
    \end{array}\right]\leq 0 \:.
\end{equation}
By substituting the interconnection~\eqref{eq:interconnUE}, we obtain~\eqref{eq:conditionL22}.

\end{proof}

Proposition \ref{propstab} allows us to fulfill \textbf{C}-1 and \textbf{C}-3 by providing a condition that ensures the closed-loop stability at the network level~$\eqref{eq:closedloopglobal}$ with a sparse and causal controller. We will address \textbf{C}-2 in the next section.

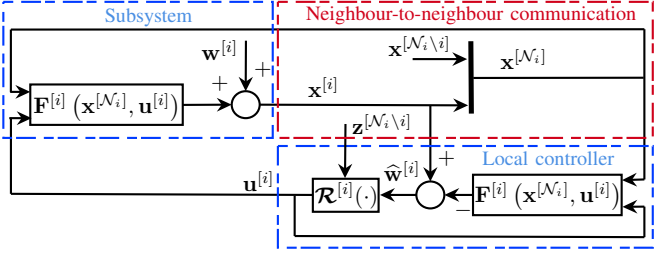
\begin{figure}
    \centering

\tikzset{every picture/.style={line width=0.75pt}} 

\begin{tikzpicture}[x=0.75pt,y=0.75pt,yscale=-0.8,xscale=0.83, thick,scale=0.85, every node/.style={scale=0.85}]

\draw   (44,95) -- (153,95) -- (153,125) -- (44,125) -- cycle ;
\draw    (152.67,110.17) -- (185,110.01) ;
\draw [shift={(188,110)}, rotate = 179.73] [fill={rgb, 255:red, 0; green, 0; blue, 0 }  ][line width=0.08]  [draw opacity=0] (10.72,-5.15) -- (0,0) -- (10.72,5.15) -- (7.12,0) -- cycle    ;
\draw   (188,110) .. controls (188,104.48) and (192.48,100) .. (198,100) .. controls (203.52,100) and (208,104.48) .. (208,110) .. controls (208,115.52) and (203.52,120) .. (198,120) .. controls (192.48,120) and (188,115.52) .. (188,110) -- cycle ;
\draw    (198.17,62.5) -- (198.01,97) ;
\draw [shift={(198,100)}, rotate = 270.25] [fill={rgb, 255:red, 0; green, 0; blue, 0 }  ][line width=0.08]  [draw opacity=0] (10.72,-5.15) -- (0,0) -- (10.72,5.15) -- (7.12,0) -- cycle    ;
\draw   (246,165) -- (293,165) -- (293,189.5) -- (246,189.5) -- cycle ;
\draw    (320,177) -- (297,177) ;
\draw [shift={(294,177)}, rotate = 360] [fill={rgb, 255:red, 0; green, 0; blue, 0 }  ][line width=0.08]  [draw opacity=0] (10.72,-5.15) -- (0,0) -- (10.72,5.15) -- (7.12,0) -- cycle    ;
\draw   (320,177) .. controls (320,171.48) and (324.48,167) .. (330,167) .. controls (335.52,167) and (340,171.48) .. (340,177) .. controls (340,182.52) and (335.52,187) .. (330,187) .. controls (324.48,187) and (320,182.52) .. (320,177) -- cycle ;
\draw    (360.5,177) -- (343,177) ;
\draw [shift={(340,177)}, rotate = 360] [fill={rgb, 255:red, 0; green, 0; blue, 0 }  ][line width=0.08]  [draw opacity=0] (10.72,-5.15) -- (0,0) -- (10.72,5.15) -- (7.12,0) -- cycle    ;
\draw   (360.5,162) -- (466.75,162) -- (466.75,192) -- (360.5,192) -- cycle ;
\draw    (29.67,100.5) -- (41,100.5) ;
\draw [shift={(44,100.5)}, rotate = 180] [fill={rgb, 255:red, 0; green, 0; blue, 0 }  ][line width=0.08]  [draw opacity=0] (10.72,-5.15) -- (0,0) -- (10.72,5.15) -- (7.12,0) -- cycle    ;
\draw    (29,120) -- (42,120) ;
\draw [shift={(45,120)}, rotate = 180] [fill={rgb, 255:red, 0; green, 0; blue, 0 }  ][line width=0.08]  [draw opacity=0] (10.72,-5.15) -- (0,0) -- (10.72,5.15) -- (7.12,0) -- cycle    ;
\draw    (330,110.33) -- (330,164) ;
\draw [shift={(330,167)}, rotate = 270] [fill={rgb, 255:red, 0; green, 0; blue, 0 }  ][line width=0.08]  [draw opacity=0] (10.72,-5.15) -- (0,0) -- (10.72,5.15) -- (7.12,0) -- cycle    ;
\draw    (483,165.5) -- (470,165.91) ;
\draw [shift={(467,166)}, rotate = 358.21] [fill={rgb, 255:red, 0; green, 0; blue, 0 }  ][line width=0.08]  [draw opacity=0] (10.72,-5.15) -- (0,0) -- (10.72,5.15) -- (7.12,0) -- cycle    ;
\draw    (483,54) -- (483,165.5) ;
\draw    (30.17,176.17) -- (245.67,177) ;
\draw    (30.13,54.25) -- (30.13,100.75) ;
\draw    (30.13,54.25) -- (483,54) ;
\draw    (30.17,119.67) -- (30.17,176.17) ;
\draw    (208,110) -- (354,110.49) ;
\draw [shift={(357,110.5)}, rotate = 180.19] [fill={rgb, 255:red, 0; green, 0; blue, 0 }  ][line width=0.08]  [draw opacity=0] (10.72,-5.15) -- (0,0) -- (10.72,5.15) -- (7.12,0) -- cycle    ;
\draw    (317.33,76.17) -- (353.33,76.17) ;
\draw [shift={(356.33,76.17)}, rotate = 180] [fill={rgb, 255:red, 0; green, 0; blue, 0 }  ][line width=0.08]  [draw opacity=0] (10.72,-5.15) -- (0,0) -- (10.72,5.15) -- (7.12,0) -- cycle    ;
\draw    (360.67,90.17) -- (482.5,90) ;
\draw  [color={rgb, 255:red, 0; green, 59; blue, 255 }  ,draw opacity=1 ][dash pattern={on 3.75pt off 3pt on 7.5pt off 1.5pt}] (221,140.67) -- (489,140.67) -- (489,216.17) -- (221,216.17) -- cycle ;
\draw  [color={rgb, 255:red, 0; green, 59; blue, 255 }  ,draw opacity=1 ][dash pattern={on 3.75pt off 3pt on 7.5pt off 1.5pt}] (24.75,33.67) -- (217.33,33.67) -- (217.33,134.67) -- (24.75,134.67) -- cycle ;
\draw    (483,185.67) -- (470.33,185.8) ;
\draw [shift={(467.33,185.83)}, rotate = 359.39] [fill={rgb, 255:red, 0; green, 0; blue, 0 }  ][line width=0.08]  [draw opacity=0] (10.72,-5.15) -- (0,0) -- (10.72,5.15) -- (7.12,0) -- cycle    ;
\draw    (233,207.5) -- (483,208.33) ;
\draw    (233,207.5) -- (233,177.5) ;
\draw    (483,208.33) -- (483,185.67) ;
\draw  [fill={rgb, 255:red, 0; green, 0; blue, 0 }  ,fill opacity=1 ] (357.9,63.68) -- (360,63.67) -- (360.36,115.83) -- (358.26,115.84) -- cycle ;
\draw    (269.33,124.33) -- (269.02,161.5) ;
\draw [shift={(269,164.5)}, rotate = 270.48] [fill={rgb, 255:red, 0; green, 0; blue, 0 }  ][line width=0.08]  [draw opacity=0] (10.72,-5.15) -- (0,0) -- (10.72,5.15) -- (7.12,0) -- cycle    ;
\draw  [color={rgb, 255:red, 208; green, 2; blue, 27 }  ,draw opacity=1 ][dash pattern={on 3.75pt off 3pt on 7.5pt off 1.5pt}] (221,33.67) -- (489,33.67) -- (489,134.67) -- (221,134.67) -- cycle ;

\draw (44.5,99.23) node [anchor=north west][inner sep=0.75pt]  [font=\normalsize]  {$\mathbf{F}^{[i]}\left( \mathbf{x}^{[\mathcal{N}_{i}]} ,\mathbf{u}^{[i]}\right)$};
\draw (245.5,166) node [anchor=north west][inner sep=0.75pt]    {$\Erre^{[ i]}( \cdot )$};
\draw (164.67,59.9) node [anchor=north west][inner sep=0.75pt]    {$\mathbf{w}^{[i]}$};
\draw (295,150) node [anchor=north west][inner sep=0.75pt]    {$\widehat{\mathbf{w}}^{[i]}$};
\draw (240.83,87.9) node [anchor=north west][inner sep=0.75pt]    {$\mathbf{x}^{[ i]}$};
\draw (171.5,87.07) node [anchor=north west][inner sep=0.75pt]    {$+$};
\draw (202.17,76.4) node [anchor=north west][inner sep=0.75pt]    {$+$};
\draw (298.33,57.07) node [anchor=north west][inner sep=0.75pt]    {$\mathbf{x}^{[\mathcal{N}_{i} \backslash i]}$};
\draw (333.5,142.4) node [anchor=north west][inner sep=0.75pt]    {$+$};
\draw (345.33,181.32) node [anchor=north west][inner sep=0.75pt]    {$-$};
\draw (95.33,36.33) node [anchor=north west][inner sep=0.75pt]   [align=left] {\textcolor[rgb]{0.29,0.56,0.89}{{\fontfamily{ptm}\selectfont {\small Subsystem}}}};
\draw (363,143) node [anchor=north west][inner sep=0.75pt]   [align=left] {\begin{minipage}[lt]{61.6pt}\setlength\topsep{0pt}
\begin{center}
\textcolor[rgb]{0.29,0.56,0.89}{{\fontfamily{ptm}\selectfont {\small Local controller}}}
\end{center}

\end{minipage}};
\draw (194,157.4) node [anchor=north west][inner sep=0.75pt]    {$\mathbf{u}^{[i]}$};
\draw (235.5,35.33) node [anchor=north west][inner sep=0.75pt]   [align=left] {\begin{minipage}[lt]{151.68pt}\setlength\topsep{0pt}
\begin{center}
{\fontfamily{ptm}\selectfont {\small \textcolor[rgb]{0.82,0.01,0.11}{Neighbour-to-neighbour communication}}}
\end{center}
\end{minipage}};
\draw (360,165.65) node [anchor=north west][inner sep=0.75pt]  [font=\normalsize]  {$\mathbf{F}^{[i]}\left(\mathbf{x}^{[\mathcal{N}_{i}]},\mathbf{u}^{[i]}\right)$};
\draw (378.33,67.07) node [anchor=north west][inner sep=0.75pt]    {$\mathbf{x}^{[\mathcal{N}_{i}]}$};
\draw (272.33,115) node [anchor=north west][inner sep=0.75pt]    {$\mathbf{z}^{[\mathcal{N}_{i} \backslash i]}$};
\end{tikzpicture}
    \caption{Proposed local controller parametrizing stabilizing controllers in terms of freely chosen sub-operator $\Erre^{[i]} \in \mathcal{L}_2$.}
    \label{fig:localContr}
\end{figure}

\subsection{Free parametrization of networked RENs}
Recently, finite-dimensional
DNN approximations of nonlinear $\mathcal{L}_p$ operators~\cite{revay2023recurrent,furieri2022neural} have been proposed.
Specifically, as suggested in~\cite{furieri2022neural} an operator $\Erre:l^{n}\rightarrow l^{m}$ is a REN if $\mathbf{u}=\mathbf{\Erre(\widehat{w})}$ is generated by the following dynamical system with input $\widehat{w}\in\mathbb{R}^n$ and output $u\in\mathbb{R}^m$:
\begin{equation} \label{eq:REN}
    \left[\begin{array}{l}
        \xi_t\\
         \nu_t \\
         u_t
    \end{array}\right] =
    \underbrace{\left[\begin{array}{lll}
        A_{1} & B_{1} & B_{2} \\
        C_{1} & D_{11} & D_{12} \\
        C_{2} & D_{21} & D_{22}
    \end{array}\right]}_{Z(\theta,\gamma)}
    \left[\begin{array}{l}
        \xi_{t-1}\\
         \sigma( \nu_t) \\
         \widehat{w}_t
    \end{array}\right], \ \ \ \xi_{-1}=0,
\end{equation}
where $\xi\in\mathbb{R}^c$ and $\nu\in\mathbb{R}^s$ are the hidden states and the input of neurons of the REN, while $\sigma:\mathbb{R}\rightarrow\mathbb{R}$ is the activation function applied element-wise. 
Further, $\sigma(\cdot)$ must be piecewise differentiable and with first derivatives restricted to the interval $[0, 1]$.
As highlighted in~\cite{revay2023recurrent}, RENs are as expressive as various existing DNN architectures. 
However, when picking any random value for their parameters, the operator $\Erre$ defined by~\eqref{eq:REN} might not lie in the space $\mathcal{L}_2$.
In~\cite{revay2023recurrent} the authors introduce a smooth function $Z$ that takes unconstrained training parameters $\theta\in\mathbb{R}^{n_{\theta}}$, for a suitable $n_{\theta}$, and a prescribed $\mathcal{L}_2$ gain $\gamma>0$ map them to a matrix $Z(\theta,\gamma)$, which defines~\eqref{eq:REN}. Notably, they ensure that the resulting operator $\Erre[\theta,\gamma]$ is in $\mathcal{L}_2$ and has $\mathcal{L}_2$-gain at most $\gamma$ for any value of the parameters $\theta$. For this reason the parametrization induced by $Z$ is termed \textit{free}.

Since our focus is on using a networked operator as explained in Section~\ref{sec:InterconOper}, we consider an interconnection of RENs that replicates the sparsity of the system. To meet requirement \textbf{C}-2 without dealing with a large-dimensional constrained optimization problem, we aim to develop a free parametrization of the networked operator. This means finding a parametrization that allows any parameter value while still meeting the constraint. To achieve this, we consider the following dynamics of the interconnected sub-operators $\Erre^{[i]}$:
\begin{equation} \label{eq:singleREN}
    \left[\begin{array}{l}
        \xi^{[i]}_t\\
         \nu^{[i]}_t \\
         z^{[i]}_t
    \end{array}\right] =
    \underbrace{\left[\begin{array}{lll}
        A_{1}^{[i]} & B_{1}^{[i]} & B_{2}^{[i]} \\
        C_{1}^{[i]} & D_{11}^{[i]} & D_{12}^{[i]} \\
        C_{2}^{[i]} & D_{21}^{[i]} & D_{22}^{[i]}
    \end{array}\right]}_{Z^{[i]}(\theta^{[i]},\gamma^{[i]})}
    \left[\begin{array}{l}
        \xi^{[i]}_{t-1}\\
         \sigma( \nu^{[i]}_t) \\
         v^{[i]}_t
    \end{array}\right], 
\end{equation}
where $\xi^{[i]}_{-1}=0$, $\forall i \in \mathcal{V}$, and 
$Z^{[i]}(\theta^{[i]},\gamma^{[i]})$ defining the free parametrization of the local REN model~\eqref{eq:singleREN}.
Before entering the details of the proposed parametrization, we consider the following Assumption on the chosen interconnection for the network of RENs.
\begin{Assumption}
\label{ass2}
We assume a particular structure for the matrices $M_{vw}$ and $M_{uz}$, more in detail: 
\begin{enumerate}[(a)]
    \item we assume that there exists a permutation matrix $C \in [0,1]^{q \times q}$ such that 
\begin{align}
    CM_{vw} = \begin{bmatrix}
I_{n}\\
0_{(q-n) \times n} 
\end{bmatrix} \: .
\end{align}
Notice that this also implies that $M_{vw}$ is a positive-valued semi-orthogonal matrix satisfying $M_{vw}^{\top} M_{vw} = I_n$.
    \item We assume that $M_{uz}$ is semi-orthogonal, i.e., it satisfies 
\begin{align}
\label{eq:h}
    M_{uz}^{\top} M_{uz} = H, \quad H=\operatorname{diag}(h), \: \: h \in \mathbb{R}^r \: .
\end{align}
    
\end{enumerate}
\end{Assumption}

\begin{rem}A few comments on Assumption \ref{ass2} are in order. 
    With (a) we allow the input of one or more sub-operators to depend on the exogenous input. More in detail, the assumption covers all cases where each sub-system can have inputs depending on different exogenous inputs. What it does not cover, is the case in which different sub-systems have inputs depending on the combination of \emph{common} exogenous inputs. We aim to relax this assumption in future work. 
    Condition (b) states that the columns of $M_{uz}$ are mutually orthogonal, which in turn implies that the output of the operator $\Erre$ is comprised of linearly independent combinations of the sub-operators output $z$. This is fairly general and gives one enough leeway on how the output $u$ is chosen. As a particular case, we have $M_{uz}=I_r$ where the output of the operator coincides with the output of all sub-operators. 
\end{rem}

Before presenting our main result, it is useful to define the set of such sub-operators that takes at least one exogenous input, which will be called \emph{out-connected}. First of all, we define
\begin{align}
      &  \mathcal{A}_{u_i} = \mbox{set of indices of $u$ associated with $i \in \mathcal{V}$}, \\
      &  \mathcal{A}_{z_i} = \mbox{set of indices of $z$ associated with $i \in \mathcal{V}$}.
\end{align}
With this, we can define
\begin{align}
\mathcal{A}_{u_i}^1 & = \left\{ k \in \mathcal{A}_{u_i} \: | \: (M_{vw} \1)_{k} =1  \right\}, \quad i \in \mathcal{V}, \\
\mathcal{A}_{u_i}^0 & = \left\{ k \in \mathcal{A}_{u_i} \: | \: (M_{vw} \1)_{k} =0  \right\}, \quad i \in \mathcal{V}.
\end{align}

These are those indices of $u_i$ that are associated with the exogenous inputs and those that are not respectively. If $\mathcal{A}_{u_i}^1$ is not empty, we say that the sub-operator $i$ is {out-connected}. The set of out-connected operators is 
\begin{align}
    \mathcal{V}^1 = \left\{ i \in \mathcal{V} \: | \: \mathcal{A}_{u_i}^1 \neq \emptyset    \right\} \: .
\end{align}

We are now able to present our main result concerning a free parametrization of distributed $\mathcal{L}_2$-stable operators $\Erre$ composed by a network of RENs.
\begin{theorem}
\label{thm:param}
   Consider a set  $\mathcal{V}=\{1, \dots, N\}$ of sub-operators $\Erre^{[i]}[\theta^{[i]},\gamma^{[i]}], \: i \in \mathcal{V}$, each dissipative with respect to the supply~\eqref{eq:supplyREN}, and interconnected according to~\eqref{eq:interconnUE}. By choosing the $\mathcal{L}_2$ gain $\gamma^{[i]}$ of each sub-operator $i \in \cV$ as follows:
\begin{align}
\label{eq:paramthm}
\resizebox{.95\hsize}{!}{$
 \gamma^{[i]} = \eta(b^{[i]}) = \begin{cases}
     & \left( \displaystyle\frac{1}{\alpha^{[i]} }\min \left\{\frac{\gamma_R^2}{\max_{j\in\mathcal{A}_{u_i}^1} \left(\sum_{k}|m_{jk}|\right) \gamma_R^2 +1}, \right. \right. \\[11pt]
& \left. \left.
\displaystyle\frac{1}{\max_{j\in\mathcal{A}_{u_i}^0} \left(\sum_{k}|m_{jk}|\right)} \right\} \right)^{\frac 12}, \: \: \mbox{if } i \in  \mathcal{V}^1  \\[14pt]
 &  \displaystyle\left(\frac{1}{\alpha^{[i]} \max_{j\in\mathcal{A}_{u_i}^0} \left(\sum_{k}|m_{jk}|\right)} \right)^{\frac 12}, \: \mbox{otherwise } 
 \end{cases} $}
\end{align}
where
\begin{align}
\label{eq:pesithm}
    \alpha^{[i]} &= h^{[i]}+ \max_{j\in\mathcal{A}_{z_i}} \left(\sum_{k}|m_{kj}|\right) + b^{[i]^2} ,  \: \forall i \in  \mathcal{V},
\end{align}
the interconnected system composed by the sub-operators $\Erre^{[i]}[\theta^{[i]},\eta(b^{[i]})]$ satisfy condition~\eqref{eq:conditionL22} for all $\theta^{[i]}, b^{[i]}$,  and the resulting operator $\Erre[\theta,\gamma_R]$ lies in $\mathcal{L}_2$.
\end{theorem}
The proof of Theorem~\ref{thm:param} is moved to the Appendix for the sake of readability. Theorem \ref{thm:param}, together with Proposition \ref{propstab}, allows us to fulfill \textbf{C}-2 by enabling an unconstrained parametrization of $\mathcal{L}_2$-stable distributed operators, effectively removing the need to enforce the hard constraint \ref{seq:globalStab} which, thanks to Proposition \ref{propstab}, led to the inequality \eqref{eq:conditionL22}. This is important as it allows the use of parametrized families of highly non-linear operators such as RENs, which can be optimized very efficiently by solving an unconstrained optimization problem over the free parameters $\theta^{[i]}, b^{[i]}$ with standard backpropagation and unconstrained gradient descent methods. 

\begin{rem}
    Notice that Theorem~\ref{thm:param} generalizes~\cite[Thm.~1]{massai2023unconstrained}. In that paper, we considered the special case where $M_{uz}^{\top}M_{uz}=I_r$ and  $M_{vw}=I_q$. The latter condition implies that $w$ can only be an additive term in each sub-operator's input $v$ (note that this also implies that $w$ is constrained to have the same dimension as the input $v$). In such a special case, we have $\mathcal{V}^1 = \cV, \mathcal{A}_{u_i}^1 = \mathcal{A}_{u_i} $ and $\mathcal{A}_{u_i}^0= \emptyset \: \forall i \in \cV$, therefore~\eqref{eq:paramthm} reduces to 
    \begin{align}
        \eta(b^{[i]})=   \sqrt{ \frac{1}{\alpha^{[i]}}\frac{\gamma_R^2}{\max_{j\in\mathcal{A}_{u_i}} \left(\sum_{k}|m_{jk}|\right) \gamma_R^2 +1}}, \: \: \forall i \in  \mathcal{V} \: ,
    \end{align}
    which is the parametrization originally proposed in \cite{massai2023unconstrained}.
\end{rem}

\medskip
A detailed description of how Theorem~\ref{thm:param} can be leveraged for the training of a network of $\mathcal{L}_{2}$-bounded RENs, is provided by Algorithm~\ref{alg:training}.

\begin{algorithm}
\caption{Training of networked $\mathcal{L}_2$ stable RENs}\label{alg:training}
\begin{algorithmic}
\State \textbf{Input:} Input data: initial conditions $\{w_0^j\}_{j=1}^{n_{exp}}$ sampled from $\mathcal{D}$, system model $ f(x,u)$, learning rate $\eta$, number of epochs $E$.
\State Initialize randomly the parameters of each subsystem: $\theta^{[i]}$, $b^{[i]}$, $i \in \cV$ and {$\gamma_R$} .

\For{$e=1$ to $E$}
\For{$j=1$ to $n_{exp}$}
    \State Initialize hidden state $\xi^{[i]}, \: \: \forall i \in \cV$.
    \State Initialize control action $u^{[i]}, \: \: \forall i \in \cV$.
        \State Compute $\gamma^{[i]}=\eta(b^{[i]}), \: \: \forall i \in \cV$ using Eq.~\eqref{eq:paramthm}.
        \For{$t=0$ to $T$}
            \State Simulate system~\eqref{eq:nonlinsys}
            \State Reconstruct the noise $\widehat{w}$
            \State Compute the inputs: 
            \State $u^{[i]}_{t} = M^{[i]}_{uz}\left( C^{[i]}_2 \xi^{[i]}_{t-1} + D^{[i]}_{21}\sigma (\nu^{[i]}_t)+D_{22}^{[i]}\widehat{w}^{[i]}_t\right)$
        \EndFor
        \EndFor
    \State Compute loss $\mathbb{E}_{\mathbf{w} \sim \mathcal{D}} \left[ \mathbf{\ell}(\mathbf{x_T,u_T}) \right]$.

    \State Backpropagate gradients through time (see~\cite{werbos1990backpropagation}) to update parameters via gradient descent with learning rate $\eta$.

\EndFor
\State \textbf{Output:} Trained model parameters $\theta_i, b^{[i]}$, $i\in\mathcal{V}$.
\end{algorithmic}
\end{algorithm}

\begin{rem}
\label{rmk:cons}
    There is usually a ``price to pay'' for transforming the problem \eqref{eq:NLPproblem} into an unconstrained one via a free parametrization like the one we derived with Theorem \ref{thm:param}. Mainly, this stems from the fact that the map $\eta$ will only output a subset of all possible $\mathcal{L}_2$-gains such that the overall operator $\Erre$ lies in  $\mathcal{L}_2$, which in turn means that we are restricting ourselves to optimize over a subset of $\mathcal{L}_{2}$-bounded operators. In particular, this comes from the use of sufficient but not necessary conditions, like Gershgorin Theorem, to ensure that \eqref{eq:conditionL22} is satisfied and the reader can take a look at the proof in the Appendix for the details. Finding less conservative methods for guaranteeing the satisfaction of \eqref{eq:conditionL22} is still a matter of ongoing research. Despite this, as it is shown in~\cite{revay2023recurrent}, the advantages of having a free parametrization greatly outweigh such a drawback.
\end{rem}

\begin{rem}
    Notice that, due to the assumption made regarding the structure of the matrices $M_{vz}$ and $M_{vw}$, the computation of each $\gamma^{[i]}$ according to~\eqref{thm:param} only necessitates knowledge of the coupling matrices $M_{uz}$ and $M_{vz}$. Therefore, by assuming suitable separability properties of the costs among agents, the proposed algorithm could be adapted to a distributed implementation.
    The distributed design of local controllers is not within the scope of this paper and will be addressed in future research.
\end{rem}

\section{Numerical example} 
In this section, we demonstrate the effectiveness of the proposed approach in addressing a complex control problem. We use PyTorch to build and train our distributed neural controller, employing stochastic gradient descent method. The implementation of our approach can be found at \url{https://github.com/DecodEPFL/Distributed_neurSLS}.
\\
The system we are considering consists of four point-mass vehicles (see Fig.~\ref{fig:example}), each with a position denoted as $ p^{[i]} \in \mathbb{R}^2 $ and velocity represented as $ v^{[i]} \in \mathbb{R}^2 $. Each vehicle experiences friction forces, which leads to the discrete-time model:
\begin{equation} \label{eq:ex_model}
\resizebox{1\hsize}{!}{$
\left[\begin{array}{l}
{p}^{[i]}_{t+1} \\
{v}^{[i]}_{t+1} \\
\end{array}\right]
=\left[\begin{array}{l}
{p}^{[i]}_{t} \\
{v}^{[i]}_{t}
\end{array}\right] + T_s
\left[\begin{array}{c}
v^{[i]}_{t} \\
m^{[i]^{-1}}
\left(-c^{[i]} |v^{[i]}_{t}|_2
+F^{[i]}_t + u^{[i]}_t
\right)
\end{array}\right],$}
\end{equation}
where, $m^{[i]}>0$ is the vehicle mass, $c^{[i]}>0$ is the friction coefficient, $u^{[i]}$ is the local control input, $T_s>0$ represents the sampling time and $F^{[i]}\in\mathbb{R}^2$ stands for the action of a local base controller whose goal is to maintain a specific formation between the agents, which is described by a fixed distance $\delta^{[i]}>0$ between them.
More specifically, the base controller stabilizes each agent around a final reference point $\bar{p}^{[i]}$, by applying the following force: 
\begin{equation}\label{eq:ex_base_contr}
    F^{[i]} = - \sum_{j=1}^{|\mathcal{N}_i\backslash i|} k^{[ij]}(|p^{[i]}-p^{[j]}|_2-\delta^{[i]}) - k^{[ir]} |\bar{p}^{[i]}-p^{[i]} |_2,
\end{equation}
where $,k^{[ij]},k^{[ir]}>0$ are the gains used in the local controllers.
The overall system obtained thus consists of nonlinear coupled subsystems~\eqref{eq:ex_model}-~\eqref{eq:ex_base_contr}. These subsystems can be written as~\eqref{eq:nonlinsyslocal}, for which the origin is a stable equilibrium.
We chose the loss function in~\eqref{eq:loss} as:
\begin{align}
\ell(\mathbf{x_T},\mathbf{u_T})&=\sum^{T}_{t=0} \ell(x_t,u_t), \\
\ell(x_t,u_t) &= \ell_{traj}(x_t,u_t)+\ell_{ca}(x_t)+\ell_{obs}(x_t)+\ell_{form}(x_t), \nonumber
\end{align}
where, $\ell_{traj}(x,u)=[x^\top \ u^\top]Q[x^\top \ u^\top]^\top$ with $Q\succeq0$ penalizes the distance between agents and their reference, as well as the magnitude of the control action. The other terms, $\ell_{ca}(x)$, $\ell_{obs}(x)$ and $\ell_{form}(x_t)$ instead, are barrier functions that penalize collisions between agents and obstacles, and incentivize maintaining the desired formation, respectively (see Fig.~\ref{fig:example}).
The expected value in the cost function~\eqref{seq:costODC} is approximated by an empirical average (see~\cite{furieri2022neural}) over $10$ randomly sampled initial conditions $\{w_0^s\}_{s=1}^{10}$.
The parameters employed in the simulation are listed in Tab~\ref{tab:parameters}.
Fig.~\ref{fig:example} shows the closed loop trajectories of the controlled system. The goal promoted by the cost $\ell(x,u)$ is to coordinate the formation in order to pass through a narrow valley avoiding collision with each others while keeping the connectivity required by the formation control. The agents start from a randomly sampled initial position marked with ``$\circ$'' and have to reach the target position marked with ``$\star$'".
\begin{figure*}[!htb]
\minipage{0.329\textwidth}
  \includegraphics[width=\linewidth]{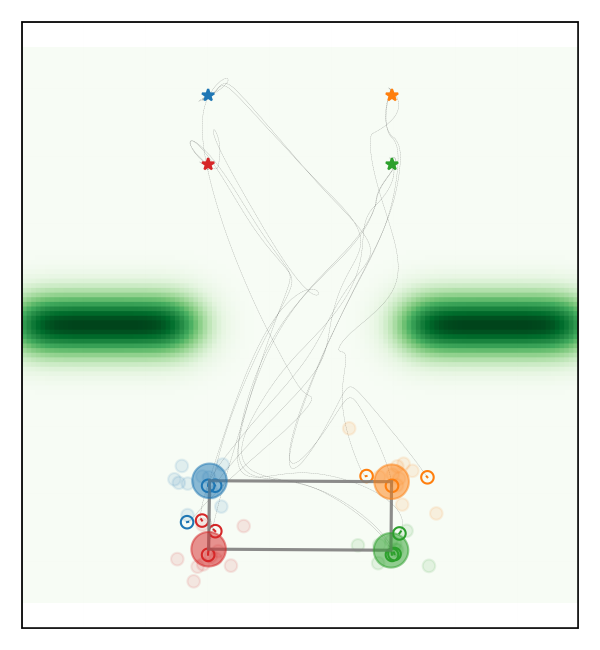}
\endminipage
\minipage{0.329\textwidth}
  \includegraphics[width=\linewidth]{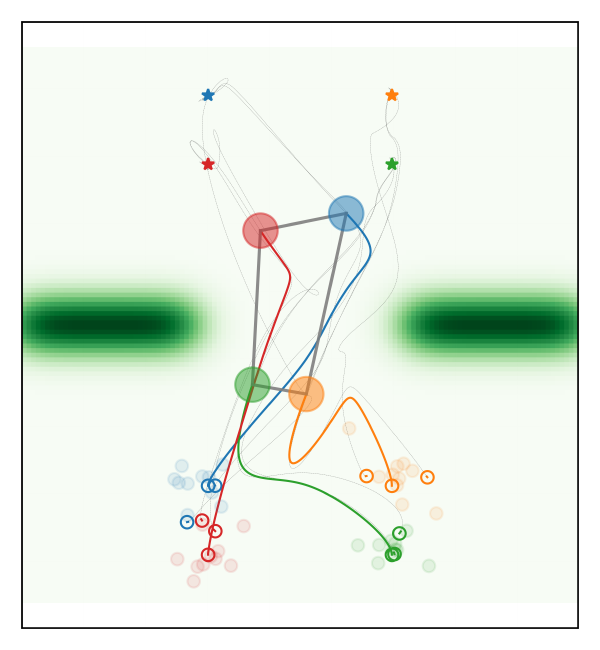}
\endminipage
\minipage{0.329\textwidth}%
  \includegraphics[width=\linewidth]{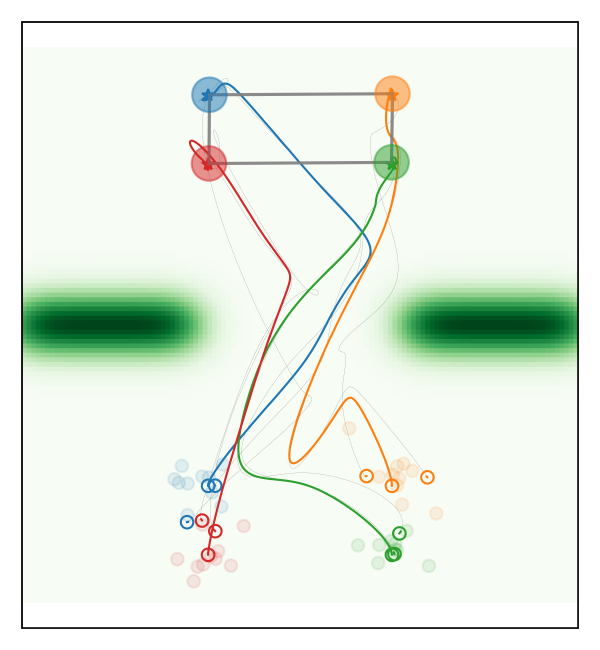}
\endminipage
 \caption{Closed-loop trajectories after training over randomly sampled initial conditions marked with ``$\circ$'' (filled transparent ``$\circ$'' are the sampled $w_0$ used for the training). Images captured at time instants $\tau_1=0.3$ $s$, $\tau_2=8.5$ $s$ and $\tau_3 = 25$ $s$. Colored (gray) lines depict the trajectories within $[0, \tau_i]$ ($[\tau_i,\infty)$). Colored circles (and their radius) represent the agents (and their size for collision avoidance), while colored ``$\star$'' symbols represent the reference position and green ellipsoids are the obstacles.}
\label{fig:example}
\end{figure*}

\begin{table}
\vspace{2mm}
\begin{center}
\begin{tabular}{ | c | c | c | c| c| c| } 
  \hline
 $m^{[i]}$& $1 \ kg$ & $c^{[i]}$& $1 \ Ns/m$ \\ 
  \hline
  $k^{[ij]}$& $ 1 \ N/m$ & $k^{[ir]}$& $ 1 \ N/m$  \\ 
  \hline
  $\delta^{[1]}=\delta^{[3]}$& $4 \ m$ & $\delta^{[2]}=\delta^{[4]}$ & $1.5 \ m$ \\ 
  \hline
  $T_s$& $0.05 \ s$ & $Q$ & blkdiag$(I_{16},0.01\cdot I_{8})$ \\ 
  \hline
  $\eta$& $0.001$ & $E$ & $1500$ \\ 
  \hline
  $\xi^{[i]}$& $25, \ \forall i \in \mathcal{V}$ & $\nu^{[i]}$ & $25$ \\ 
  \hline
\end{tabular} \caption{List of parameters employed in the simulation $\forall i,j\in\mathcal{V}$.}
\label{tab:parameters}
\end{center}
\end{table}
\section{Conclusions}
This paper introduces a novel distributed control method for interconnected nonlinear systems. By harnessing the framework of neural system level synthesis and dissipativity theory, we achieve optimized performance for user-defined cost functions while guaranteeing stability throughout the training process  at the global network level. A key feature is the ability to train local control policies using unconstrained optimization algorithms. We hope that the proposed framework will pave the way for applications in several domains where achieving both stability and optimal performance has traditionally been difficult, such as smart-grid control~\cite{nahata2020passivity}, flocking control for autonomous vehicles~\cite{carron2023multi}, and distributed optimization algorithm design \cite{martin2024learning}. Further important theoretical developments include the addition of noisy outputs, unknown system dynamics, and parallelizing the training across agents.

\appendix
\subsection{Proof of Theorem~\ref{thm:param}} \label{appendix}
Firstly we rewrite $\mathbf{X}(\alpha^{[i]}X^{[i]})$ in a more compact form as:
\begin{align}
    \mathbf{X}(\alpha^{[i]}X^{[i]})=&\left[\begin{array}{ll}
\Pi_{N,v} & 0\\
0 & -\Pi_{N,z} 
\end{array}\right] \:,
\end{align}
with $\Pi_{N,v}=\text{blkdiag}\{\alpha^{[1]}\gamma^{[1]^2} I_{q_1},\dots,\alpha^{[N]}\gamma^{[N]^2} I_{q_N}\}$ and  $\Pi_{N,z}=\text{blkdiag} \{ \alpha^{[1]}I_{r_1},\dots,\alpha^{[N]}I_{r_N} \}$. 
By Proposition \ref{propstab},~\eqref{eq:conditionL22} needs to be satisfied, yielding:
\begin{align}
\left[\begin{smallmatrix}
-M_{vz}^{\top}\Pi_{N,v} M_{vz} +\Pi_{N,z}-M_{uz}^\top M_{uz} & -M_{vz}^{\top} \Pi_{N,v} M_{vw} \\[2ex]
-M_{vw}^\top \Pi_{N,v} M_{vz} & - M_{vw}^\top \Pi_{N,v} M_{vw}+\gamma_R^2 I_n
\end{smallmatrix}\right] \succ 0   \:.
\end{align}
By applying the Schur complement, the above condition is equivalent to 
\begin{align}
   & -M_{vz}^{\top}\Pi_{N,v} M_{vz} +\Pi_{N,z}-M_{uz}^\top M_{uz} \notag \\
   & -M_{vz}^{\top} \Pi_{N,v} M_{vw}[- M_{vw}^\top \Pi_{N,v} M_{vw} \label{eq:schur1} \\  \notag
    & +\gamma_R^2 I_n]^{-1} M_{vw}^\top \Pi_{N,v}  M_{vz} \succ 0  \\[2ex] 
   & - M_{vw}^\top \Pi_{N,v}M_{vw}+\gamma_R^2 I_n \succ 0 \:. \label{eq:schur2}
\end{align}
Using the structure of $M_{vw}$ given by Assumption~\ref{ass2}, easy computations show that $M_{vw}^\top \Pi_{N,v} M_{vw} = P_1 \in \mathbb{R}^{n \times n}$ is a diagonal matrix where the diagonal elements are given by $\alpha^{[i]} \gamma^{[i]^2}, \: \forall i \in \mathcal{V}^1 $, each one repeated as many times as the cardinality of $\mathcal{A}_{u_i}^1$.
Therefore, condition~\eqref{eq:schur2} reads as
\begin{align}\label{eq:shurcond}
- M_{vw}^\top \Pi_{N,v} M_{vw}+\gamma_R^2 I_n & \succ 0 \\
-P_1+\gamma_R^2 I_n  & \succ 0 \nonumber \\
   \iff  -\alpha^{[i]} \gamma^{[i]^2} + \gamma_R^2 & > 0 \nonumber  \\
 \alpha^{[i]} \gamma^{[i]^2} & < {\gamma_R^2}, \forall i \in \mathcal{V}^1  
  \:. \label{eq:firstcond}
\end{align}
As for condition \eqref{eq:schur1} 
by collecting $M_{vz}$ we obtain
\begin{multline}
    M_{vz}^{\top} (-\Pi_{N,v}-\Pi_{N,v}  M_{vw}[- M_{vw}^\top \Pi_{N,v} M_{vw}\\
    +\gamma_R^2 I_n]^{-1} M_{vw}^\top \Pi_{N,v}  ) M_{vz} +\Pi_{N,z}-M_{uz}^\top M_{uz} \succ 0
\end{multline}
By applying the Schur complement again this is equivalent to:
\begin{align} 
     \left[\begin{array}{ll}
\Pi_{N,z}-H & M_{vz}^{\top} \\[2ex]
M_{vz} & -D^{-1}
\end{array}\right] & \succ 0 \label{eq:shur3}  \\[2ex]
 D & \prec 0 \:, \label{eq:shur4}
\end{align}
where we recall that $M_{uz}^\top M_{uz}=H$ is a diagonal matrix by Assumption~\ref{ass2} and
\begin{multline}
    D = (-\Pi_{N,v}-\Pi_{N,v}  M_{vw}[- M_{vw}^\top \Pi_{N,v} M_{vw} \\
    +\gamma_R^2 I_n]^{-1} M_{vw}^\top \Pi_{N,v} ) \:.
\end{multline}
Using again the structure of $M_{vw}$ given in Assumption \ref{ass2}, it is easy to show that $ M_{vw}[- M_{vw}^\top \Pi_{N,v} M_{vw} \\
    +\gamma_R^2 I_n]^{-1} M_{vw}^\top = P_2 \in \mathbb{R}^{q \times q}$ is a diagonal matrix such that $(P_2)_{jj} = (-\alpha^{[i]} \gamma^{[i]^2} + \gamma_R^2)^{-1}$ if $j \in \mathcal{A}_{u_i}^1, \:  i \in \mathcal{V}^1$ and $(P_2)_{jj} =0 $ otherwise. From this, we can see that $D$ is a diagonal matrix given by
\begin{align}
\label{eq:D}
    (D)_{jj} = \begin{cases}
        \displaystyle\frac{\alpha^{[i]} \gamma^{[i]^2}  \gamma_R^2}{\alpha^{[i]} \gamma^{[i]^2} -\gamma_R^2}  \quad  & \mbox{for } j  \in \mathcal{A}_{u_i}^1, \:  i \in \mathcal{V}^1 \\
        -\alpha^{[i]} \gamma^{[i]^2}  \quad  & \mbox{for } j  \in \mathcal{A}_{u_i}^0, \:  i \in \mathcal{V} \:.
    \end{cases}
\end{align}
Notice that condition \eqref{eq:shur4} is always satisfied when \eqref{eq:firstcond} holds. Finally, we enforce \eqref{eq:shur3} by applying the Gershgorin Theorem~\cite{salas1999gershgorin}.
In particular, let us focus on the first $p$ rows of the matrix~\eqref{eq:shur3}.
We want all the Gershgorin circles to lie in the non-negative half-plane, i.e.
\begin{align}
\label{eq:paramproof}
    \underbrace{\left(\Pi_{N,z}-H\right)_{jj}}_{\mbox{Center}} -\underbrace{\sum_{k}|{\left (M_{vz}^{\top}\right)_{jk}}|}_{\mbox{Radius}} \ge 0, \quad  \forall j \in  \mathcal{A}_{z_i},  i \in \cV.
\end{align}
Notice that the Gershgorin center is $\alpha^{[i]}-1$ for every $j \in \mathcal{A}_{z_i}$, hence condition \eqref{eq:paramproof} can be enforced by imposing $\alpha^{[i]}$ such that:
\begin{align}
\label{eq:pesi}
    \alpha^{[i]} &= h^{[i]}+ \max_{j\in\mathcal{A}_{z_i}} \left(\sum_{k}|m_{kj}|\right) + b^{[i]^2} ,  \: \forall i \in  \mathcal{V},
\end{align}
where $b^{[i]} \in \mathbb{R}$.
Similarly, we can impose the other Gershgorin circles associated with the remaining rows of the matrix~\eqref{eq:shur3} to lie in the non-negative half of the plane. Using~\eqref{eq:D}, this yields the conditions:
\begin{align}
\label{eq:paramproof3}
0 &\leq \alpha^{[i]} \gamma^{[i]^2} \leq \frac{\gamma_R^2}{\max_{j\in\mathcal{A}_{u_i^w}} \left(\sum_{k}|m_{jk}|\right) \gamma_R^2 +1}, \: \forall i \in  \mathcal{V}^{1}   \\
0 &\leq \alpha^{[i]} \gamma^{[i]^2} \leq \frac{1}{\max_{j\in\mathcal{A}_{u_i^0}} \left(\sum_{k}|m_{jk}|\right)}, \: \forall i \in  \mathcal{V} \:,
\end{align}
notice that \eqref{eq:paramproof3} ensures that \eqref{eq:firstcond} holds true.
The conditions above can be rewritten as
\begin{align}
\label{eq:paramproof4}
0 &\leq \alpha^{[i]} \gamma^{[i]^2} \leq \min \left\{\frac{\gamma_R^2}{\max_{j\in\mathcal{A}_{u_i}^1} \left(\sum_{k}|m_{jk}|\right) \gamma_R^2 +1}, \right. \notag \\
& \left.
\frac{1}{\max_{j\in\mathcal{A}_{u_i}^0} \left(\sum_{k}|m_{jk}|\right)} \right\}, \: \: \forall i \in  \mathcal{V}^1  \\
0 &\leq \alpha^{[i]} \gamma^{[i]^2} \leq \frac{1}{\max_{j\in\mathcal{A}_{u_i}^0} \left(\sum_{k}|m_{jk}|\right)}, \: \forall i \in  \mathcal{V} \setminus \mathcal{V}^{1} \:, 
\end{align}
where $\alpha^{[i]}$ is given by \eqref{eq:pesi}. From this, we readily obtain the parametrization of the $\mathcal{L}_2$ gains given in~\eqref{eq:paramthm}.
\hfill$\blacksquare$

\bibliographystyle{IEEEtran}
\bibliography{biblio}    

\end{document}